\documentclass[12pt,reqno]{article}
\usepackage{amsmath,amssymb,amsfonts,amscd,latexsym,amsthm,mathrsfs}
\usepackage[unicode]{hyperref}
\newtheorem{lemma}{Lemma}
\newtheorem{theorem}{Theorem}

\theoremstyle{remark}

\newcommand{\qbin}[2]{\genfrac{[}{]}{0pt}{}{#1}{#2}}
\renewcommand{\pmod}[1]{\;\allowbreak(\operatorname{mod}#1)}
\begin{document}

\title{The method of creative microscoping}

\author{Wadim Zudilin\thanks{This work is supported by JSPS Invitational Fellowships for Research in Japan (fellowship S19126).
The text is based on my talk at the conference ``Analytic Number Theory and Related Topics''
(Research Institute for Mathematical Sciences, Kyoto University,
Japan, 15--18~October 2019).}\\[1mm]{\small Department of Mathematics, IMAPP, Radboud University,}\\[-1.2mm]
{\small PO Box 9010, 6500~GL Nijmegen, Netherlands}}

\date{December 2019}

\maketitle

\begin{abstract}
% We outline basic principles of a new method that gives a conceptual reasoning for and, at the same time, proofs of (super)congruences for truncated sums of arithmetic hypergeometric evaluations.
Ramanujan's formulae for $1/\pi$ and their generalisations remain an amazing topic, with many mathematical challenges.
Recently it was observed that the formulae possess spectacular `supercongruence' counterparts.
For example, truncating the sum in Ramanujan's formula
$$
\sum_{k=0}^\infty\frac{\binom{4k}{2k}{\binom{2k}{k}}^2}{2^{8k}3^{2k}}\,(8k+1)=\frac{2\sqrt{3}}{\pi}
$$
to the first $p$ terms correspond to the congruence
$$
\sum_{k=0}^{p-1}\frac{\binom{4k}{2k}{\binom{2k}{k}}^2}{2^{8k}3^{2k}}\,(8k+1)\equiv p\biggl(\frac{-3}p\biggr)\pmod{p^3}
$$
valid for any prime $p>3$.
Some supercongruences were shown to be true through a tricky use of classical hypergeometric identities or the Wilf--Zeilberger method of creative telescoping.
The particular example displayed above (and many other entries) were resistant to such techniques.
In joint work with Victor Guo we develop a new method of `creative microscoping' that provides conseptual reasons for and  simultaneous proofs of both the underlying Ramanujan's formula and its finite supercongruence counterparts.
The main ingredient is an asymptotic analysis of suitable $q$-deformations of Ramanujan's formulas at all roots of unity.
Here we outline the method on the concrete example of supercongruence above.
\end{abstract}

%==================================================

More than a century ago \cite{Ra14}, Ramanujan produced a list of rapidly convergent series to $1/\pi$.
Those give one a practical way to calculate $\pi$ with high accuracy, a way tested to be highly efficient in the computer era.
Mathematically, Ramanujan's formulae were analysed and proven by a diverse range of methods, and numerous generalisations were found through their study.
A typical example on Ramanujan's original list is the following identity.

\begin{theorem}[Ramanujan]
\label{th1}
The following equality is true\textup:
\begin{equation}
\sum_{k=0}^\infty\frac{\binom{4k}{2k}{\binom{2k}{k}}^2}{2^{8k}3^{2k}}\,(8k+1)
=\frac{2\sqrt{3}}{\pi}.
\label{ram1}
\end{equation}
\end{theorem}

The series on the left-hand side in \eqref{ram1} is an instance of hypergeometric series, a series of the form $\sum_{k\in\mathbb Z}R(k)$ for which $R(x+1)/R(x)$ is a rational function of variable~$x$.
A canonical \emph{hypergeometric} form of writing the series uses Pochhammer's symbol (also known as shifted factorial)
$$
(a)_k=\frac{\Gamma(a+k)}{\Gamma(a)}=\prod_{j=0}^{k-1}(a+j)
\quad\text{for}\; k=0,1,2,\dotsc.
$$
Then identity \eqref{ram1} can be equivalently stated as
\begin{equation}
\sum_{k=0}^\infty\frac{(\frac14)_k(\frac12)_k(\frac34)_k}{k!^3\,9^k}\,(8k+1)
=\frac{2\sqrt{3}}{\pi}.
\label{ram1-a}
\end{equation}
The corresponding function
$$
R(x)=\frac{\Gamma(\frac14+x)\Gamma(\frac12+x)\Gamma(\frac34+x)}{\Gamma(\frac14)\Gamma(\frac12)\Gamma(\frac34)\Gamma(1+x)^3}\,\frac{8x+1}{9^x}
$$
vanishes at $x=-1,-2,\dots$, so that the sum over $\mathbb Z$ reduces to the one over nonnegative integers $k$.

One of remarkable arithmetic features of \emph{all} Ramanujan's formulae for $1/\pi$ is presence of companion supercongruences: truncation of the corresponding infinite sum at the $(p-1)$-th place happens to be a simple expression modulo $p^3$ for all but finitely many primes $p$.
This was observed by this author \cite{Zu09} following earlier hints of L.~Van Hamme \cite{VH97}.
In our example \eqref{ram1}--\eqref{ram1-a}, the result reads as follows.

\begin{theorem}
\label{th2}
We have
\begin{equation}
\sum_{k=0}^{p-1}\frac{(\frac14)_k(\frac12)_k(\frac34)_k}{k!^3\,9^k}\,(8k+1)
\equiv p\biggl(\frac{-3}p\biggr)\pmod{p^3}
\quad\text{for $p>3$ prime},
\label{ram1a}
\end{equation}
where the Jacobi--Kronecker symbol $\bigl(\frac{-3}{\cdot}\bigr)$ `replaces' the square root of $3$.
\end{theorem}

Observe that the only primes that show up in the denominator of the left-hand side
$$
\sum_{k=0}^{p-1}\frac{(\frac14)_k(\frac12)_k(\frac34)_k}{k!^3\,9^k}\,(8k+1)
=\sum_{k=0}^{p-1}\frac{\binom{4k}{2k}{\binom{2k}{k}}^2}{2^{8k}3^{2k}}\,(8k+1)
$$
are 2 and 3, so that the congruence is well understood as the identity in $\mathbb Z/p^3\mathbb Z$ for any prime $p>3$.

The last two decades witnessed development of methods to prove supercongruences like \eqref{ram1a}, sometimes modulo a smaller power of $p$ and usually on a case-by-case study.
Apart from involvement of classical (and semi-classical) congruences, those methods require somewhat tricky applications of numerous hypergeometric identities or make use of the algorithm of creative telescoping, for instance, of suitable Wilf--Zeilberger pairs.
These strategies did not allow one to establish Theorem~\ref{th2} but they hinted at existence of a $q$-counterpart of the arithmetic phenomena behind Ramanujan's formulae and alike.
In the remaining part of this review we outline how $q$-analogues of Theorems~\ref{th1} and \ref{th2} look like, how we prove them in the joint work \cite{GZ19} with V.~Guo and how they are used to establish the theorems.

In order to state those $q$-analogues we need to familiarise ourselves with standard $q$-hypergeometric notation.
We deal with $q$ inside the unit disc, $|q|<1$, and define
$$
(a;q)_\infty=\prod_{j=0}^\infty(1-aq^j)
$$
and the related $q$-Pochhammer symbol by
$$
(a;q)_k=\frac{(a;q)_\infty}{(aq^k;q)_\infty}=\prod_{j=0}^{k-1}(1-aq^j)
\quad\text{for}\; k=0,1,2,\dots,
$$
for non-negative integers $n$, so that
$$
\lim_{q\to1}\frac{(q^a;q)_n}{(1-q)^n}=(a)_n
\quad\text{and}\quad
\lim_{q\to1}\frac{(q;q)_\infty(1-q)^{1-a}}{(q^a;q)_\infty}=\Gamma(a).
$$
Furthermore, we introduce the $q$-numbers and $q$-binomial coefficients as
$$
[n]=[n]_q=\frac{1-q^n}{1-q}
\quad\text{and}\quad
\qbin nm={\qbin nm}_q=\frac{(q;q)_n}{(q;q)_m(q;q)_{n-m}},
$$
so that they correspond to $n$ and $\binom nm$, respectively, in the limit as $q\to1$.
This notation is used to record basic (or $q$-) hypergeometric series, which are of the form $\sum_{k\in\mathbb Z}R(k)$ with $R(x+1)/R(x)$ being a rational function of~$q^x$.

\begin{theorem}[$q$-analogue of Theorem~\ref{th1}]
\label{th3}
The following equality is true\textup:
\begin{equation}
\sum_{k=0}^{\infty}\frac{(q;q^2)_k^2 (q;q^2)_{2k}}{(q^2;q^2)_{2k}(q^6;q^6)_k^2}[8k+1]q^{2k^2}
=\frac{(q^3;q^2)_\infty (q^3;q^6)_\infty }{(q^2;q^2)_\infty (q^6;q^6)_\infty}.
\label{q4}
\end{equation}
\end{theorem}

\begin{theorem}[$q$-analogue of Theorem~\ref{th2}]
\label{th4}
Let $n$ be a positive integer coprime with $6$. Then
\begin{equation}
\sum_{k=0}^{n-1}\frac{(q;q^2)_k^2 (q;q^2)_{2k}}{(q^2;q^2)_{2k}(q^6;q^6)_k^2}[8k+1]q^{2k^2}
\equiv q^{-(n-1)/2}[n]\biggl(\frac{-3}{n}\biggr) \pmod{[n]\Phi_n(q)^2}.
\label{q4a}
\end{equation}
\end{theorem}

In the last theorem, the truncated $q$-hypergeometric sums are considered modulo (products of) cyclotomic polynomials
$$
\Phi_n(q)=\prod_{\substack{j=1\\ \gcd(j,n)=1}}^n(q-e^{2\pi ij/n})\in\mathbb Z[q].
$$
Notice that $[n]_q=\prod_{d\mid n,\,d>1}\Phi_d(q)$ and that $[p]_q=\Phi_p(q)\to p$ as $q\to1$ when $p$~is prime.
We use the convention to interpret the congruence $A_1(q)/A_2(q)\equiv0\pmod{P(q)}$ for polynomials $A_1(q),A_2(q),P(q)\in\mathbb Z[q]$ as follows: $P(q)$ divides $A_1(q)$ and is coprime with $A_2(q)$;
more generally, $A(q)\equiv B(q)\pmod{P(q)}$ for rational functions $A(q),B(q)\in\mathbb Z(q)$ means $A(q)-B(q)\equiv0\pmod{P(q)}$.

\begin{proof}[Proof of Theorems~\textup{\ref{th1}} and~\textup{\ref{th2}}]
In the case of formula~\eqref{q4}, we see that
\begin{gather*}
\lim_{q\to1}\frac{(q;q^2)_{2k}}{(q^2;q^2)_{2k}}
=\frac{(\frac12)_{2k}}{(2k)!}=\frac{(\frac14)_k(\frac34)_k}{(\frac12)_k\,k!},
\\
\lim_{q\to1}\frac{(q;q^2)_k}{(q^6;q^6)_k}
=\lim_{q\to1}\frac{(q;q^2)_k}{(q^2;q^2)_k\,\prod_{j=1}^k(1+q^{2j}+q^{4j})}
=\frac{(\frac12)_k}{k!\,3^k}
\\ \intertext{and}
\lim_{q\to1}\frac{(q^3;q^6)_\infty}{(q^6;q^6)_\infty(1-q^6)^{1/2}}
=\lim_{q\to1}\frac{(q;q^2)_\infty}{(q^2;q^2)_\infty(1-q^2)^{1/2}}
=\frac{1}{\Gamma(\tfrac12)}=\frac{1}{\sqrt\pi},
\end{gather*}
hence in the limit as $q\to1$ we obtain \eqref{ram1a}.
At the same time, taking the limit as $q\to1$ in \eqref{q4a} for $n=p$ prime leads to the Ramanujan-type supercongruences \eqref{ram1a}.
\end{proof}

Our proof of Theorem~\ref{th4} combines two principles. One corresponds to achieving the congruences in \eqref{q4a} modulo $[n]$ only, and for this we deal with the $q$-hypergeometric sum \eqref{q4} at a `$q$-microscopic' level\,---\,that is, at roots of unity (and this cannot be transformed into a derivation of \eqref{ram1a} directly from~\eqref{ram1}).
Another `creative' principle is about getting more parameters involved in the $q$-story.

\begin{theorem}
\label{th5}
Let $n$ be a positive integer coprime with $6$.
Then, for any indeterminates $a$ and $q$, we have modulo $[n](1-aq^n)(a-q^n)$,
\begin{equation}
\sum_{k=0}^{n-1}\frac{(aq;q^2)_k (q/a;q^2)_k (q;q^2)_{2k}}{(q^2;q^2)_{2k}(aq^6;q^6)_k (q^6/a;q^6)_k }[8k+1]q^{2k^2}
\equiv q^{-(n-1)/2}[n]\biggl(\frac{-3}{n}\biggr).
\label{q4a-new}
\end{equation}
\end{theorem}

\begin{proof}[Proof of Theorem \textup{\ref{th4}}]
The denominator of \eqref{q4a-new} related to $a$ is the factor $(aq^6;q^6)_{n-1}\*(q^6/a;q^6)_{n-1}$;
its limit as $a\to1$ is relatively prime to $\Phi_n(q)$, since $n$ is coprime with~$6$.
On the other hand, the limit of $(1-aq^n)(a-q^n)$ as $a\to1$ has the factor $\Phi_n(q)^2$.
Thus, letting $a\to1$ in \eqref{q4a-new} we see that \eqref{q4a} is true modulo $\Phi_n(q)^3$.
At the same time, by considering \eqref{q4a-new} modulo $[n]$ only and specialising $a=1$ in the result reads
\begin{equation*}
\sum_{k=0}^{n-1}\frac{(q;q^2)_k^2 (q;q^2)_{2k}}{(q^2;q^2)_{2k}(q^6;q^6)_k^2}[8k+1]q^{2k^2}
\equiv0\pmod{[n]}.
\end{equation*}
Thus, indeed both sides of \eqref{q4a} are congruent modulo $[n]\Phi_n(q)^2$.
\end{proof}

In turn, the general set of congruences in Theorem~\ref{th5} is deduced from a nonterminating version of \eqref{q4a-new}.

\begin{theorem}
\label{th6}
The following identity is true\textup:
\begin{equation}
\sum_{k=0}^\infty\frac{(aq;q^2)_k (q/a;q^2)_k (q;q^2)_{2k}}{(q^2;q^2)_{2k}(aq^6;q^6)_k (q^6/a;q^6)_k }[8k+1]q^{2k^2}
=\frac{(q^5;q^6)_\infty(q^7;q^6)_\infty(aq^3;q^6)_\infty(q^3/a;q^6)_\infty}
{(q^2;q^6)_\infty(q^4;q^6)_\infty(aq^6;q^6)_\infty(q^6/a;q^6)_\infty}.
\label{q4t}
\end{equation}
\end{theorem}

\begin{proof}
In fact, there is not so much to prove here since the equality is a special case of the summation
\begin{align*}
&
\sum_{k=0}^\infty\frac{(1-acq^{4k})(a;q)_k(q/a;q)_k (ac;q)_{2k}}
{(1-ac)(cq^3;q^3)_k(a^2cq^2;q^3)_k(q;q)_{2k}}\,q^{k^2}
\\ &\qquad
=\frac{(acq^2;q^3)_\infty(acq^3;q^3)_\infty(aq;q^3)_\infty(q^2/a;q^3)_\infty}
{(q;q^3)_\infty(q^2;q^3)_\infty(a^2cq^2;q^3)_\infty(cq^3;q^3)_\infty}.
\end{align*}
which is the limiting $d\to0$ case of the cubic transformation due to Rahman \cite[Eq.~(3.8.18)]{GR04}.
Replacing $q$ with $q^2$, taking $c=q/a$ and then $aq$ for $a$ one arrives at~\eqref{q4t}.
\end{proof}

\begin{proof}[Proof of Theorem~\textup{\ref{th3}}]
Take $a=1$ in \eqref{q4t}.
\end{proof}

In the remainder we discuss the most nontrivial part of the method of creative telescoping\,---\,deduction of Theorem~\ref{th5} from Theorem~\ref{th6}.

\begin{lemma}
\label{lem-new}
Let $n$ be a positive odd integer. Then
\begin{equation}
\sum_{k=0}^{n-1}\frac{(q^{1-n};q^2)_k (q^{1+n};q^2)_k (q;q^2)_{2k}}{(q^2;q^2)_{2k}(q^{6-n};q^6)_k (q^{6+n};q^6)_k }[8k+1]q^{2k^2}
=q^{-(n-1)/2}[n]\biggl(\frac{-3}{n}\biggr).
\label{eq:lem3.1}
\end{equation}
\end{lemma}

\begin{proof}
We substitute $a=q^{n}$ into \eqref{q4t}.
Then the left-hand side of \eqref{q4t} terminates at $k=(n-1)/2$, hence also equals the sum up to $n-1$ featured
on the left-hand side of \eqref{eq:lem3.1}.
On the other hand, the substitution transforms the right-hand side of \eqref{q4t} into
$q^{-(n-1)/2}[n]$ if $n\equiv1\pmod 3$, into $-q^{-(n-1)/2}[n]$ if $n\equiv2\pmod 3$, and into 0 if $3\mid n$.
\end{proof}

\begin{proof}[Proof of Theorem~\textup{\ref{th5}}]
Let $\zeta\ne1$ be a primitive $d$-th root of unity, where $d\mid n$ and $n>1$ is coprime with 6.
Denote by
$$
c_q(k)=[8k+1]\frac{(aq;q^2)_k(q/a;q^2)_k (q;q^2)_{2k}}{(q^2;q^2)_{2k}(aq^6;q^6)_k(q^6/a;q^6)_k}\,q^{2k^2}
$$
the $k$-th term of the sum \eqref{q4t} and write \eqref{q4t} as
\begin{equation}
\sum_{\ell=0}^{\infty}c_q(\ell d)\sum_{k=0}^{d-1}\frac{c_q(\ell d+k)}{c_q(\ell d)}
=\frac{(q;q^2)_\infty(q^6;q^6)_\infty(aq^3;q^6)_\infty(q^3/a;q^6)_\infty}
{(1-q)\,(q^2;q^2)_\infty(q^3;q^6)_\infty(aq^6;q^6)_\infty(q^6/a;q^6)_\infty}.
\label{q4u}
\end{equation}
Consider the limit as $q\to\zeta $ radially, that is, $q=r\zeta$ where $r\to1^-$.
On the left-hand side we get
$$
\lim_{q\to\zeta}\frac{c_q(\ell d+k)}{c_q(\ell d)}
=\frac{c_\zeta(\ell d+k)}{c_\zeta(\ell d)}
=c_\zeta(k)
$$
and
$$
\lim_{q\to\zeta}c_q(\ell d)
=\frac1{2^{4\ell}}\binom{4\ell}{2\ell}
\frac{(a\zeta;\zeta^2)_{\ell d}(\zeta/a;\zeta^2)_{\ell d}}{(a\zeta^6;\zeta^6)_{\ell d}(\zeta^6/a;\zeta^6)_{\ell d}}
=\frac1{2^{4\ell}}\binom{4\ell}{2\ell}.
$$
By Stirling's approximation, $\binom{4\ell}{2\ell}\sim2^{4\ell}/\sqrt{2\pi\ell}$ as $\ell\to\infty$, hence
$$
\sum_{\ell=0}^\infty\frac1{2^{4\ell}}\binom{4\ell}{2\ell}=\infty.
$$
For the right-hand side of \eqref{q4u},
\begin{align*}
&\lim_{q\to\zeta}\frac{(q;q^2)_{\ell d+k}(q^6;q^6)_{\ell d+k}(aq^3;q^6)_{\ell d+k}(q^3/a;q^6)_{\ell d+k}}
{(1-q)\,(q^2;q^2)_{\ell d+k}(q^3;q^6)_{\ell d+k}(aq^6;q^6)_{\ell d+k}(q^6/a;q^6)_{\ell d+k}}
\\
&\qquad
=\frac{(\zeta;\zeta^2)_k(\zeta^6;\zeta^6)_k(a\zeta^3;\zeta^6)_k(\zeta^3/a;\zeta^6)_k}
{(1-\zeta)\,(\zeta^2;\zeta^2)_k(\zeta^3;\zeta^6)_k(a\zeta^6;\zeta^6)_k(\zeta^6/a;\zeta^6)_k}
\end{align*}
for any $\ell\ge0$ and $0\le k<d$. Therefore, the expression on the right-hand side of \eqref{q4u}
is bounded above by
$$
\max_{0\le k<d}\frac{|(\zeta;\zeta^2)_k(\zeta^6;\zeta^6)_k(a\zeta^3;\zeta^6)_k(\zeta^3/a;\zeta^6)_k|}
{|(1-\zeta)\,(\zeta^2;\zeta^2)_k(\zeta^3;\zeta^6)_k(a\zeta^6;\zeta^6)_k(\zeta^6/a;\zeta^6)_k|}+1
$$
as $q\to\zeta$. By comparing the two asymptotics we conclude that
\begin{equation*}
\sum_{k=0}^{d-1}c_\zeta(k)=0
\end{equation*}
in turn implying that
\begin{equation*}
\sum_{k=0}^{n-1}c_\zeta(k)
=\sum_{k=0}^{d-1}c_\zeta(k)+\sum_{k=d}^{2d-1}c_\zeta(k)+\dots+\sum_{k=n-d}^{n-1}c_\zeta(k)
=\frac nd\sum_{k=0}^{d-1}c_\zeta(k)
=0.
\end{equation*}
This means that $\sum_{k=0}^{n-1}c_q(k)\equiv0\pmod{\Phi_d(q)}$ for any $d\mid n$, $d>1$, hence
\begin{equation*}
\sum_{k=0}^{n-1}\frac{(aq;q^2)_k (q/a;q^2)_k (q;q^2)_{2k}}{(q^2;q^2)_{2k}(aq^6;q^6)_k (q^6/a;q^6)_k }[8k+1]q^{2k^2}
\equiv0\equiv q^{-(n-1)/2}[n]\biggl(\frac{-3}{n}\biggr) \pmod{[n]}.
\end{equation*}
On the other hand, it follows from Lemma~\ref{lem-new} that
$$
\sum_{k=0}^{n-1}\frac{(aq;q^2)_k (q/a;q^2)_k (q;q^2)_{2k}}{(q^2;q^2)_{2k}(aq^6;q^6)_k (q^6/a;q^6)_k }[8k+1]q^{2k^2}
=q^{-(n-1)/2}[n]\biggl(\frac{-3}{n}\biggr)
$$
when $a=q^n$ or $a=q^{-n}$.
This implies that the congruences
$$
\sum_{k=0}^{n-1}\frac{(aq;q^2)_k (q/a;q^2)_k (q;q^2)_{2k}}{(q^2;q^2)_{2k}(aq^6;q^6)_k (q^6/a;q^6)_k }[8k+1]q^{2k^2}
\equiv q^{-(n-1)/2}[n]\biggl(\frac{-3}{n}\biggr)
$$
hold true modulo $1-aq^n$ and $a-q^n$.
Since $[n]$, $1-aq^n$ and $a-q^n$ are relatively prime polynomials, we obtain \eqref{q4a-new}.
\end{proof}

It is worth mentioning that the congruences in Theorem~\ref{th2}, \ref{th4} and \ref{th5} remain true when the sums are truncated at $(p-1)/2$ or $(n-1)/2$, respectively. These other(!) companion congruences are also settled in \cite{GZ19} by the method of creative telescoping.

\medskip
We can summarise our derivation path of the results as follows:
\begin{align*}
\text{Theorem~\ref{th6}}
& \;\;\underset{a=1}\Longrightarrow\;\; \text{Theorem~\ref{th3}} \;\;\underset{q\to1}\Longrightarrow\;\; \text{Theorem~\ref{th1}}
\\[-6pt]
\Downarrow\qquad &
\\
\text{Theorem~\ref{th5}}
& \;\;\underset{a=1}\Longrightarrow\;\; \text{Theorem~\ref{th4}} \;\;\underset{q\to1}\Longrightarrow\;\; \text{Theorem~\ref{th2}}
\end{align*}
The top of this scheme\,---\,Theorem~\ref{th6}\,---\,comes essentially for free from the book~\cite{GR04} nicknamed the $q$-Bible among the specialists in combinatorics and hypergeometric functions.
Many further entries from \cite{GR04} lead to remarkable (and quite difficult!) congruences, so that the $q$-Bible turns out to be a treasury book for number theory as well.

More recent work of V.~Guo, some in collaboration with M.~Schlosser, myself and others, extends the horizons of applicability of creative telescoping even further.
One of the latest achievements is a general framework (of $q$-analogues) of Dwork-type supercongruences~\cite{Guo19}.
It is already clear, though not well accepted by those who still have fun playing tricks with classical congruences, that the method is capable of proving more than we can even expect at this moment.
In addition, the method links the congruences, through asymptotic behaviour of $q$-series at roots of unity, to recent studies of mock theta functions and quantum modular forms.
It also shares similarities with the method from \cite{BTT18} that discusses algebraic relations of finite multiple zeta values via those available for $q$-analogues of multiple zeta values.

\bigskip
\noindent
\textbf{Acknowledgements.}
I am thankful to the organizers Masatoshi Suzuki and Takashi Nakamura
of the RIMS conference ``Analytic Number Theory and Related Topics''
(Kyoto University, Japan, 15--18~October 2019)
for invitation to give a talk at the meeting.
Special thanks go to my host Yasuo Ohno and his team from the Tohoku University (Sendai);
they made my stay in Japan both culturally and scientifically enjoyable.

I am incredibly indebted to Victor J.W.~Guo for permanent fruitful conversations on the subject and inspiration.

\end{document}